\documentclass[a4paper,11pt]{amsart}

\usepackage{hyperref,upgreek,mathtools,stmaryrd,enumitem,bbm} 

\usepackage{mathabx}

\usepackage[textsize=footnotesize,color=yellow!70, bordercolor=white]{todonotes}

\usepackage{xcolor}
\definecolor{dblue}{rgb}{0,0,0.70}
\hypersetup{
	unicode=true,
	colorlinks=true,
	citecolor=dblue,
	linkcolor=dblue,
	anchorcolor=dblue
}

\usepackage
[a4paper, margin=1.2in]{geometry}

\usepackage[latin1]{inputenc}
\usepackage[T1]{fontenc}
\usepackage[english]{babel}

\usepackage{mathrsfs}
\usepackage{amscd}
\usepackage{amsfonts}
\usepackage{amsmath}
\usepackage{amssymb}
\usepackage{amstext}
\usepackage{amsthm}
\usepackage{amsbsy}

\usepackage{xspace}
\usepackage[all]{xy}
\usepackage{graphicx}
\usepackage{url}
\usepackage{latexsym}

\makeatletter
\newcommand*{\rom}[1]{\expandafter\@slowromancap\romannumeral {\sharp}1@}
\makeatother

\theoremstyle{definition}

\newcommand{\pow}{\mathcal{P}}

\newcommand{\WO}{\mathrm{WO}}

\newcommand{\ck}{\mathrm{ck}}

\renewcommand{\Col}{\mathrm{Col}}

\catcode`\<=\active \def<{
\fontencoding{T1}\selectfont\symbol{60}\fontencoding{\encodingdefault}}
\catcode`\>=\active \def>{
\fontencoding{T1}\selectfont\symbol{62}\fontencoding{\encodingdefault}}

\newcommand{\tmdate}[1]{\today}

\newcommand{\NN}{\mathbb{N}}

\newtheorem{fact}{Fact}[section]

\newtheorem{theorem}[fact]{Theorem}
\newtheorem*{theorem*}{Theorem}
\newtheorem{lemma}[fact]{Lemma}

\newtheorem{corollary}[fact]{Corollary}
\newtheorem{definition}[fact]{Definition}

\newtheorem{question}[fact]{Question}

\newtheorem*{question*}{Question}
\newtheorem*{problem*}{Problem}
\newtheorem*{generalproblem*}{General Problem}
\newtheorem*{problem A}{Problem 1}
\newtheorem*{problem B}{Problem 2}
\newtheorem*{claim*}{Claim}

\theoremstyle{remark} 
\newtheorem{remark}[fact]{Remark}

\newenvironment{enumerate-(a)}{\begin{enumerate}[label={\upshape (\alph*)}, leftmargin=2pc]}{\end{enumerate}}
\newenvironment{enumerate-(a)-r}{\begin{enumerate}[label={\upshape (\alph*)}, leftmargin=2pc,resume]}{\end{enumerate}}
\newenvironment{enumerate-(A)}{\begin{enumerate}[label={\upshape (\Alph*)}, leftmargin=2pc]}{\end{enumerate}}
\newenvironment{enumerate-(A)-r}{\begin{enumerate}[label={\upshape (\Alph*)}, leftmargin=2pc,resume]}{\end{enumerate}}
\newenvironment{enumerate-(i)}{\begin{enumerate}[label={\upshape (\roman*)}, leftmargin=2pc]}{\end{enumerate}}
\newenvironment{enumerate-(i)-r}{\begin{enumerate}[label={\upshape (\roman*)}, leftmargin=2pc,resume]}{\end{enumerate}}
\newenvironment{enumerate-(I)}{\begin{enumerate}[label={\upshape (\Roman*)}, leftmargin=2pc]}{\end{enumerate}}
\newenvironment{enumerate-(I)-r}{\begin{enumerate}[label={\upshape (\Roman*)}, leftmargin=2pc,resume]}{\end{enumerate}}
\newenvironment{enumerate-(1)}{\begin{enumerate}[label={\upshape (\arabic*)}, leftmargin=2pc]}{\end{enumerate}}
\newenvironment{enumerate-(1)-r}{\begin{enumerate}[label={\upshape (\arabic*)}, leftmargin=2pc,resume]}{\end{enumerate}}

\begin{document}





\author{Merlin Carl}
\address{Europa-Universit\"at Flensburg, 
Institut f\"ur mathematische, naturwissenschaftliche und technische Bildung
Abteilung f\"ur Mathmematik und ihre Didaktik
Geb\"aude Riga 1, 
Auf dem Campus 1b, 
24943 Flensburg} 
\email{merlin.carl@uni-flensburg.de}
\urladdr{}

\author{Philipp Schlicht}
\address{School of Mathematics, University of Bristol, Fry Building, Woodland Road, Bristol, BS8 1UG, UK
and 
Institute for Mathematics, University of Vienna, Kolingasse 14-16, 1090 Vienna, Austria} 
\email{philipp.schlicht@bristol.ac.uk }

\author{Philip Welch}
\address{School of Mathematics, University of Bristol, Fry Building, Woodland Road, Bristol, BS8 1UG, UK} 
\email{p.welch@bristol.ac.uk}

\thanks{We thank our anonymous referee for several remarks that helped to clarify and improve the exposition of our work.} 
\thanks{This project has received funding from the European Union's Horizon 2020 research and innovation programme under the Marie Sk{\l}odowska-Curie grant agreements No 794020 of the second-listed author (Project \emph{IMIC: Inner models and infinite computations}). He was also partially supported by FWF grant number I4039.}

\title{Decision times of infinite computations} 
\date{\today}

\begin{abstract} 
The \emph{decision time} of an infinite time algorithm is the supremum of its halting times over all real inputs. 
The \emph{decision time} of a set of reals is the least decision time of an algorithm that decides the set; semidecision times of semidecidable sets are defined similary. 
It is not hard to see that $\omega_1$ is the maximal decision time of sets of reals. 
Our main results determine the supremum of countable decision times as $\sigma$ and that of countable semidecision times as $\tau$, where $\sigma$ and $\tau$ denote the suprema of $\Sigma_1$- and $\Sigma_2$-definable ordinals, respectively, over $L_{\omega_1}$. 
We further compute analogous suprema for singletons. 
\end{abstract} 

\maketitle

\thispagestyle{plain} 

\setcounter{tocdepth}{2}

\section{Introduction}

Infinite time Turing machines (ittm's) were invented by Hamkins and Kidder as a natural machine model allowing a standard Turing machine to operate not only through unboundedly many finite stages, but  transfinitely, thus passing through an $\omega$'th stage and beyond. 
They link computability, descriptive set theory and low levels of the constructible hierarchy. 

While the analogy pursued by Hamkins and Lewis \cite{hamkins2000infinite} was that of Turing reducibility and its degree theory, the notion of recursion most closely analogous to it in the literature turned out to be that of Kleene recursion (see for example \cite{Hi78} for an account of this).
In this theory, $\Pi^1_1$ sets of integers were characterised by transfinite processes bounded by $\omega_1^{\ck}$ in time, and could be viewed as resulting from computation calls (either viewed as arising from systems of equations, or from Turing machines) along wellfounded computable trees. 

This theory can be reformulated by modeling transfinite processes by ittm's. 
\cite{hamkins2000infinite} showed that the subsets, either of the natural numbers or Baire space, which are computable by the basic ittm's, fall strictly between the $\Pi^1_1$ and $\Delta^1_2$ sets.
They thus provide natural classes in this region and thereby form a natural test case for properties of classes in the Wadge hierarchy. 

The computational properties of infinite time Turing machines lead to some new phenomena that do not occur in Turing machines. 
For instance, there are two natural notions of forming an output -- besides a program halting in a final state with an element of Cantor space on its output tape, one can consider the `eventual output', if it occurs, when the output tape is seen to stabilise even though the machine has not formally halted, but is perhaps just working away on its scratch tape. In terms of the machine architecture one could argue that this eventual output is characteristic: to analyse the halting times of ittm's, one has in any case to analyse these stabilisation times, and halting is a special case of stabilisation.

Another new phenomenon is the appearance of new reals on the output or work tapes beyond all 
stabilisation times in processes which do not stabilise at all. What are these reals? Hamkins and Lewis dubbed such reals {\em accidental}. They are constructed in some process, written to the output tape, but are evanescent: later they may be overwritten and disappear.

We make this clearer by giving a brief sketch how these machines work (for more details see \cite{hamkins2000infinite}.) 
An {\em ittm-program} is  just a regular Turing program. 
The `hardware' of an ittm consists of an input, work and output tape, each a sequence of cells of order type $\omega$, and a single head that may move one cell to the left or right along the tapes each of which thus have an initial leftmost cell, and are infinite in the rightward direction. The read/write head can read, say, the $n$'th cell from each of the three tapes simultaneously, if the head is situated in position $n$. 
(The three tapes are convenient for defining stabilisation, but a single tape model has the same computational strength.)
Each cell contains $0$ or $1$.  At time $\alpha$, the machine proceeds to  $\alpha+1$ by following the ordinary Turing program, and acts depending on what it sees on the tape at its current position and the program dictates, just as an ordinary Turing machine of such a kind would. 
However at any limit stage $\mu$ of time, by fiat, the contents of every cell on the tape is set to the inferior limit or \emph{liminf} of the earlier contents at times $\alpha < \mu$ of this cell. Thus a `$1$' is in a cell at stage $\mu$ if and only if $\exists \beta <\mu \forall \alpha \in (\beta,\mu)$ that cell has a `$1$'. The head is positioned at the liminf of its positions at times $\alpha <\mu$, and the current state or instruction number to the liminf again of those prior to $\mu$. 

Note that in \cite{hamkins2000infinite} an inessentially different but equivalent set of rules is obtained: limsups rather than liminfs were used
for cell values, the read/write head was returned to position zero, and a special limit state 
was entered. This set of rules can be seen to provide the same class of computable functions. Other variations are possible. It is the liminf (or limsup) rule on the cell values that is the determining feature. Its $\Sigma_{2}$ nature is complete in the sense that all other possible choices of $\Sigma_{2}$-definable rules are reducible to it by \cite[Theorem 2.9]{W}. 

These machines define the following classes of sets. 
We loosely call any element of Baire space, Cantor space or $\pow(\omega)$ a \emph{real}. 
Let $\chi_A$ denote the characteristic function of a subset $A$ of $\pow(\omega)$. 
Such a set $A$ is called \emph{ittm-decidable} if and only if there is an ittm-program $p$ such that $p(x)$ halts with output $\chi_A(x)$ for any subset $x$ of $\omega$. 
We shall often omit the prefix \emph{ittm}. 
$A$ is called \emph{semidecidable} if and only if there is an ittm-program $p$ such that $p(x)$ halts precisely if $x\in A$. 
We say that $A$ is \emph{cosemidecidable} if its complement is semidecidable. 
For singletons $A=\{x\}$, we call $x$ \emph{recognisable} if $A$ is decidable. 
Further, $x$ will be called \emph{semirecognisable} if $\{x\}$ is semidecidable, and cosemirecognisable if $\{x\}$ is cosemidecidable. 
This terminology 
is analogous to 
automata theory and 
recognisable languages. 
From a  logician's perspective, one 
might call a real $x$ implicitly definable if the singleton $A=\{x\}$ is definable. 
In the `lost melody theorem' of \cite{hamkins2000infinite}, the divergence 
between implicitly and explicitly definable reals is studied. 
This phenomenon does not appear in Turing machines. 
But it is not a new phenomenon:  
for example, the real  
$0^{(\omega)}$ 
coding all arithmetical truths is implicitly definable by a $\Pi^{0}_{2}$ definition, whilst obviously 
it cannot be explicitly defined at any finite level of the arthmetical hierarchy.

By a result of the third-listed author \cite[Theorem 1.1]{welch2000length}, the supremum of ittm-halting times on empty input equals the supremum $\lambda$ of writable ordinals, {\em i.e.} those ordinals for which an ittm can compute a code on zero input. 
This solved a well known problem posed by Hamkins and Lewis. 
We consider an extension of Hamkins' and Lewis' problem by allowing all real inputs. 
More precisely, we consider the problem which decision times of sets and singletons are possible. 
This is defined as follows. 

\newpage 

\begin{definition} \ 
\begin{enumerate-(a)} 
\item 
The \emph{decision time} of a program $p$ is the supremum of its halting times for arbitrary inputs. 
\item 
The \emph{decision time} of a decidable set $A$ is the least decision time of a program that decides $A$. 
\item 
The \emph{semidecision time} of a semidecidable set $A$ is the least decision time of a program that semidecides $A$. 
\item 
The \emph{cosemidecision time} of a cosemidecidable set $A$ is the semidecision time of its complement.
\end{enumerate-(a)} 
\end{definition} 






Since halting times are always countable, it is clear that these ordinals are always at most $\omega_{1}$. It is also not hard to see that the bound $\omega_{1}$ is attained (see Lemma \ref{long decision times} below). 
The questions are then: 
Which countable ordinals can occur as decision or semidecision times of sets of real numbers? 
Which ordinals occur for singletons? 
As there are only countable many programs, there are only countably many semidecidable sets of real numbers, so the suprema of the countable (semi-)decision times must be countable. 
Writing $\omega_1$ for $\omega_1^V$ (throughout the paper), we shall show that these
suprema can be determined in the course of this paper. (We call a supremum \emph{strict} to emphasise that it is not attained.)

\begin{definition} \ 
\begin{enumerate-(a)} 
\item 
$\sigma$ denotes the first $\Sigma_1$-stable ordinal, {\em i.e.} the least $\alpha$ with $L_\alpha\prec_{\Sigma_1}L_{\omega_1}$.
Equivalently, this is the supremum of the $\Sigma_1^{L_{\omega_{1}}}$-definable ordinals (see Lemma \ref{versions of sigma}).\footnote{We mean that for an ordinal $\alpha$, the set $\{\alpha\}$ is $\Sigma_1$-definable over $L_{\omega_1}$, or equivalently in $V$. The reason for writing $L_{\omega_1}$ here is to make the analogy with the definition of $\tau$ clear. } 
\item 
$\tau$ denotes the supremum of the $\Sigma_{2}^{L_{\omega_{1}}}$-definable ordinals. 
\end{enumerate-(a)} 
\end{definition} 

While the value of $\sigma$ is absolute, we would like to remark that $\tau$ is sensitive to the underlying model of set theory. 
For instance, $\tau<\omega_1^L$ holds in $L$, but $\tau>\omega_n^L$ if $\omega_n^L$ is countable in $V$. 
Note that $\tau$ equals the ordinal $\gamma^1_2$ studied in \cite{MR1011178} by a result of \cite{countableranks}. 

The following are our main results: 

\begin{theorem*} (see Theorem \ref{countable decision times}) 
The strict supremum of countable decision times for sets of reals equals $\sigma$. 
\end{theorem*} 


\begin{theorem*} (see Theorem \ref{supremum of countable semidecision times}) 
The strict supremum of countable semidecision times for sets of reals equals $\tau$.  
\end{theorem*} 


\begin{theorem*} (see Theorems \ref{sigma upper set clockable bound}, \ref{supremum of semidecision times of semirecognisable singletons} and \ref{supremum of cosemidecision times of reals}) 
The strict suprema of decision times, semidecision times and cosemidecision times for singletons equal $\sigma$. 
\end{theorem*} 


We also prove the existence of semidecidable and cosemidecidable singletons that are not recognisable. 
The latter answers \cite[Question 4.5.5]{carl2019ordinal}. 

Since there are gaps in the clockable ordinals (see \cite{hamkins2000infinite}), it is natural to ask whether there are gaps below $\sigma$ in the countable decision times. 
We answer this in Theorem \ref{uniform running time bound} by showing that gaps of arbitrarily large lengths less than $\sigma$ exists.

\section{Preliminaries}

We fix some notation and recall some facts. 

The symbols $x,y,z,\ldots$ will be reserved for reals and elements of Cantor space ${}^{\omega}2$. 
$\WO$ will denote the set of reals that code strict total orders of $\omega$ which are wellordered.   This class  is a complete $\Pi^{1}_{1}$ set of reals. It is a basic result of \cite[Cor. 2.3]{hamkins2000infinite} that is ittm-decidable.

As usual in admissible set theory (see \cite{Bar}), we write $\omega_{1}^{x}=\omega_{1}^{x,\ck}$ for the least ordinal not recursive in $x$. It is thus the least ordinal $\alpha$ so that $L_{\alpha}[x]$ is an admissible set. We use the well-known fact that if $M$ is a transitive model which is a union of admissible sets that are elements of $M$, then $\Sigma^{1}_{1}$ relations are absolute to $M$. Thus if $\varphi(y)$ is a $\Sigma^{1}_{1}$ statement about $y\in M$, then $\varphi(y)\Leftrightarrow \varphi(y)^{M}$. 

When $\alpha$ is an ordinal, let $\alpha^\oplus$ denote the least admissible ordinal that is strictly above $\alpha$. When $\alpha$ is an ordinal which is countable in $L$, then $x_\alpha$ denotes the $<_{L}$-least real coding $\alpha$. 
Conversely, when a real  $x$, regarded as a set of integers, codes an ordinal via a recursive pairing function on $\omega$, then we denote this ordinal by $\alpha_x$. 

We write $p(x){\downarrow}y$ if a program $p$ with input $x$ halts with $y$ on its output tape, and $p(x){\downarrow}^{\leq\alpha}y$ if it  so halts at or before time $\alpha$. 
Moreover, $y$ can be omitted if one does not want to specify the output. 

\begin{definition} \ 
Each of the following is defined as the supremum of ordinals coded by reals of the following form: 
\begin{enumerate-(a)} 
\item 
$\lambda$ for halting outputs of some $p(0)$. 
\item 
$\zeta$ for stable outputs of some $p(0)$. 
\item 
$\Sigma$ for reals which appear on the tape of some computation $p(0)$. 
\end{enumerate-(a)} 
\end{definition}

These three classes are then the \emph{writable}, the \emph{eventually writable}, and the \emph{accidentally writable} ordinals respectively.

The relativisations for reals $x$ are denoted $\lambda^x$, $\zeta^x$ and $\Sigma^x$.  
In particular $\omega_{1}^{x}$ is smaller than $\lambda^{x}$. 
From this one can show the following fact, which we shall use without further mention in the sequel. 
If  $p(x){\downarrow}y$, then we say that $y$ is {\em ittm-computable} from $x$ and write $y\leq_{\infty}x$. We further write $x=_{\infty}y$ if $x\leq
_{\infty} y $ and $ y\leq_{\infty} x$. 

\begin{lemma} 
\label{L_lambdax is writable-invariant}  
$y \leq_{\infty}x \Leftrightarrow y\in L_{\lambda^{x}}[x] \Leftrightarrow L_{\lambda^{y}}[y] \subseteq L_{\lambda^{x}}[x]$.
\end{lemma}

We shall use the following characterisation of $\lambda$, $\zeta$ and $\Sigma$: 

\begin{theorem} 
\label{lambda-zeta-Sigma-theorem} 
 {\em (The $\lambda$-$\zeta$-$\Sigma$ Theorem; {\em cf.} \cite{W}, \cite[Corollary 32]{welch2009characteristics}) \\ 
The triple  $(\lambda,\zeta,\Sigma)$ is the lexicographically least triple so that $L_{\lambda}\prec
_{\Sigma_{1}}L_{\zeta}\prec_{\Sigma_{2}}L_{\Sigma}$. }
\end{theorem} 

The action of any particular program $p$ 
on input an integer, will of course depend on the program itself, but there are programs $p$ for which $p(k)$ 
does not halt but runs for ever. The significance of the ordinals $\zeta$ and $\Sigma $ in this case is that by stage $\zeta$, the machine will enter a \emph{final loop} from $\zeta$ to $\Sigma$. 
In a final loop from $\alpha$ to $\alpha+\beta$, the snapshots of the machine at stage $\alpha+(\beta\cdot \gamma)$ are by definition identical for all ordinals $\gamma$. 
In other words, the loop is repeated endlessly. 
(A computation may have a loop that is iterated finitely often without it being of this final looping kind.)
The point is that one can easily recognise final loops from $\alpha$ to $\alpha+\beta$ as the loops with the following property for each cell: if the inferior limit $1$ is attained at $\alpha+\beta$, then the contents of the cell is constant throughout the loop. 

Of particular interest is the \emph{Theory Machine} ({\em cf.} \cite{friedman2007two}) that also does not stop, but on zero input writes (and overwrites) the $\Sigma_{2}$-theories of the levels of the $L_{\alpha}$ hierarchy to the output tape for all $\alpha \leq\Sigma$. But the $\Sigma_{2}$-Theory of $L_{\Sigma}$ is identical to that of $L_{\zeta}$. Hence 
there is a final loop from $\zeta$ to $\Sigma$, but no shorter final loop and none beginning before $\zeta$. 

In this  `$\lambda$-$\zeta$-$\Sigma$-Theorem' one should note that $\Sigma$ is a limit of admissible ordinals, but is not itself admissible. As part of the analysis of this theorem, $\gamma$, the supremum of {\em halting times} of computations $p(k)$ 
on integer   input (the \emph{clockable} ordinals) was shown to be $\lambda$ (\cite{welch2000length}). The ordinals which have a code computed as an output of some $p(k)$ 
(the \emph{writable} ordinals) are an initial segment of the countable ordinals, but the clockables are not.

\begin{definition} \ 
$\sigma_\nu$ denotes the supremum of $\Sigma_1^{L_{\omega_1}}$-definable ordinals with parameters in $\nu\cup\{\nu\}$. 
\end{definition}

\begin{lemma}
\label{versions of sigma} 
$L_{\sigma_\nu}\prec_{\Sigma_1} L_{\omega_1}$ for any countable ordinal $\nu$. 
\end{lemma} 
\begin{proof} 
Assume $\nu=0$ for ease of notation. 
Let $\hat{\sigma}$ denote the least $\alpha$ with $L_\alpha\prec_{\Sigma_1} L_{\omega_1}$. 
It suffices to show that every element of $L_{\hat{\sigma}}$ is $\Sigma_1^{L_{\omega_1}}$-definable. 
If not, then the set $N$ of $\Sigma_1^{L_{\omega_1}}$-definable elements of $L_{\omega_1}$ is not transitive, so the collapsing map $\pi\colon N\rightarrow \bar{N}$ moves some set $x$. 
Assume that $x$ has minimal $L$-rank and $x$ is $\Sigma_1^{L_{\omega_1}}$-definable by $\varphi(x)$. 
One may check that $N\prec_{\Sigma_1}L_{\omega_1}$: suppose $(\exists z\ \psi(z, x_{0},\ldots \, , x_{n}))^{L_{\omega_{1}}}$ with $x_{i}\in
N$ and $\psi$ a $\Sigma_{0}$-formula. Then the $<_L$-least such $z$ is $\Sigma_{1}$-definable in the vector of the $x_{i}$; replacing each $x_{i}$ by its $\Sigma_{1}$-definition, yields a $\Sigma_{1}$ definition of such a $z$. Hence $(\exists z\ \psi(z, x_{0},\ldots \, , x_{n}))^N$.
Thus $\bar{N}\models \varphi(\pi(x))$. 
Since $\bar{N}$ is transitive, this implies $V\models\varphi(\pi(x))$. 
Since $\pi(x)\neq x$, this contradicts the assumption that $\varphi(x)$ has $x$ as its unique solution. 
\end{proof} 

Note that $\Sigma_1$-statements in $H_{\omega_1}$ are equivalent to $\Sigma^1_2$-statements and conversely \cite[Lemma 25.25]{Je03}. 
In particular, $L_\sigma$ is $\Sigma^1_2$-correct in $V$. 
We shall use this without further reference below. 


\section{Decision times}

In this section, we focus on 
the problem of ascertaining the decision times of ittm-semidecidable sets. 

\subsection{The supremum of countable decision times} 

By a standard condensation argument, (see \cite[Thm.1.1]{hamkins2000infinite}) halting times of ittm's on arbitrary inputs are always countable, so it is clear that decision times are always at most $\omega_1$. 

Recall that, by \cite[Corollary 2.3]{hamkins2000infinite}, all $\Pi^1_1$ sets are ittm-decidable. 

\begin{lemma}
\label{long decision times} 
Every set with countable decision time is Borel. 
Hence any non-Borel $\Pi^1_1$ set has decision time $\omega_1$. 
\end{lemma}
\begin{proof} 
Suppose that an ittm-program $p$ semidecides a set $A$ within a countable time $\alpha$. 
Note that $p(x){\downarrow}^{\leq\alpha}$ can be expressed by $\Sigma^1_1$ and $\Pi^1_1$ formulas in any code for $\alpha$. 
By Lusin's separation theorem, it is Borel. 
\end{proof} 

For instance, the set $\WO$ of wellorders on the natural numbers is $\Pi^1_1$-complete and hence not Borel. Its decision time thus equals $\omega_1$.
Since $\Pi_1^1$-sets are ittm-decidable, $\WO$ is ittm-decidable with decision time $\omega_1$ (cf. \cite[Prop. 32]{carl2020space}).

It remains to study sets with countable decision times, and in particular, the following question: 

\begin{question*} 
What is the supremum of \textit{countable} decision times of sets of reals? 
\end{question*} 



We need two auxiliary results to answer this problem. 
The next lemma shows that if $x$ is semi-recognisable, but $x\notin L_{\alpha^\oplus}$, then $\{x\}$'s semidecision  time is greater than $\alpha$. 

\begin{lemma} \
\label{recognisable reals appear quickly} 
\begin{enumerate-(1)} 
\item
\label{recognisable reals appear quickly 1} 
If $p$ semirecognises $x$ and $p(x){\downarrow}^{\leq\alpha}$, then $x\in L_{\alpha^\oplus}$. 
\item 
\label{recognisable reals appear quickly 2} 
The bound of $\alpha^{\oplus}$  is in general optimal.
\end{enumerate-(1)} 
\end{lemma} 
\begin{proof} 
\ref{recognisable reals appear quickly 1} 
Let $M=L_{\alpha^\oplus}$ and take any $\Col(\omega,\alpha)$-generic filter $g\in V$ over $M$. 
Since $\Col(\omega,\alpha)$ is a set forcing in $M$, $M[g]$ is admissible if $g$ is taken to be sufficiently generic. (By \cite[Theorem 10.17]{mathias2015provident}, it suffices that the generic filter meets every dense class that is a union of a $\Sigma_1$-definable with a $\Pi_1$-definable class.)  In this model everything is countable. Let $y\in \WO\cap M[g]$ be a real coding $\alpha$. 

Set $R(z)$ if ``$\exists h\ [ h $ {\em codes a  sequence of computation snapshots of $p(z)$, along the
ordering $y$, which converges with output $1$}]''. Then as $ x \in R$, the latter is a non-empty $\Sigma^{1}_{1}(y)$ predicate; by an effective $\Sigma^{1}_{1}$ Perfect Set Theorem (see \cite[III Thm.6.2]{Sa90})  (relativized to $y$) if there is no solution to $R$ in $L_{\omega_1^y}[y]$ 
then there is a perfect set of such solutions in $V$.  But $R=\{x\}$. Hence $x\in  L_{\omega_{1}^y}[y] = M[g]$. 
As $\Col(\omega,\alpha)$ is homogeneous ({\em cf.}  \cite[Corollary 26.13]{Je03}) we can see that $x\in  L_{\alpha^\oplus}$ by asking for each $n\in\omega$ whether $\Col(\omega,\alpha)$ forces $n$ to be in some real $z$ such that $p(z)$ halts. 




\ref{recognisable reals appear quickly 2} 
This is essentially \cite[Thm LII]{Rog}: to sketch why this is so, take any computable ordinal $\gamma< \omega^{\ck}_{1}$. 
One can construct a real $x$ which is $\Pi^0_2$ as a singleton and codes a sequence of iterated (ordinary Turing jumps) of length $\omega\cdot \gamma$. 
Then $x \notin L_\gamma$, (as the theory of $L_{\gamma}$ is reducible to $x$), but $x$ is recognisable in time $\omega +1$ since $x$ is $\Pi^0_2$ (recall, for example, that on zero input a complete $\Pi^{0}_{2}(x)$ set can be written to the worktape in $\omega$ stages). 
This shows that $\omega^\oplus=\omega^{\ck}_{1}$ is optimal when $\alpha=\omega +1$. 
\end{proof}



The next result will be used to provide a lower bound for decision times of sets. 

\begin{theorem}
\label{sigma upper set clockable bound}
The supremum of decision times of singletons equals $\sigma$. 
\end{theorem}

\begin{remark} 
It should be unsurprising that the supremum of decision times is at least $\sigma$. 
It is well known that the $\Pi^{1}_{1}$ singletons are wellordered and appear unboundedly in $L_\sigma$ by work of Suzuki \cite{Su64}, and these are clearly ittm-decidable. Moreover, their order type is $\sigma$ (see e.g. \cite[Exercise 16.63]{Rog}). 
\end{remark} 

\begin{proof}
To see that the supremum is at least $\sigma$, take any $\alpha<\sigma$. 
Pick some $\beta$ with $\alpha<\beta<\sigma$ such that 
some $\Sigma_1$-sentence $\varphi$ holds in $L_\beta$ for the first time.
We claim that the $<_{L}$-minimal code $x$ for $L_{\beta^{\oplus}}$ is recognisable. 
To see this, let $T$ denote the theory $\mathsf{KP}+(V=L)+\varphi$. 
Devise a program $p(z)$ that checks if $z$ codes a wellfounded model of $T+$``{\em there is no transitive model of $T$}'' and halts if so. 
However, 
such a code $x$ is not an element of $L_{\beta^{\oplus}}$.
By Lemma \ref{recognisable reals appear quickly}, the decision time of $\{x\}$ is thus at least $\beta^{\oplus}$. 

It remains to show that any recognisable real $x$ is recognisable with a uniform time bound strictly below $\sigma$.
To see this, suppose that $p$ recognises $x$. 
We shall run $p$ and a new program $q$ synchronously, and halt as soon as one of them does. 
Thus $q$ ensures that the halting time is small. 

We now describe $q$. 
A run 
$q(y)$ simulates all ittm-programs with input $y$ synchronously. 
For each halting output on one of these tapes, we check whether it codes a linear order. 
In this case, run a wellfoundedness test and save the wellfounded part, as far as it is detected. 
(These routines are run synchronously for all tapes, one step at a time.) 
A wellfoundedness test works as follows. 
We begin by searching for a minimal element; this is done by a subroutine that searches for a strictly decreasing sequence $x_0, x_1, \dots$.
If the sequence cannot be extended at some finite stage, we have found a minimal element and add it to the wellfounded part. 
The rest of the algorithm is similar and proceeds by successively adding new elements to the wellfounded part. 
Each time the wellfounded part increases to some $\alpha+1$ by adding a new element, 
we construct a code for $L_{\alpha+1}$. 
(Note that the construction of $L_\alpha$ takes approximately $\omega^{\omega}\cdot \alpha$ many steps\footnote{Roughly, this can be seen as follows: Given a code for $\alpha$, split the tape into $\alpha$ many disjoint portions of length $\omega$ and construct the code level-wise. To pass from level $\xi$ to level $\xi+1$ requires computing the sets $\{a\in L_{\xi}:L_{\xi}\models\phi(a,p)\}$ for each $\in$-formula $\phi$ and each parameter $p\in L_{\xi}$. When $\phi$ is $\Sigma_n$, this can be done in $\omega^n+1$ many steps. Doing this for all pairs $(\phi,p)$ - which can be arranged in order type $\omega$ using our code -- can thus be done with time bound $\omega^\omega$.}.) 
We then search for $z$ such that $p(z){\downarrow}^{\leq\alpha}$
$1$ in $L_{\alpha+1}$. 
We halt if such a $z$ is found and $x\neq z$. 

By Lemma \ref{recognisable reals appear quickly}, $x\in L_{\lambda^x}$. 
So for any $y$ with $\lambda^y\geq\lambda^x$, some $L_\alpha$ satisfying \emph{$p(x){\downarrow}^{\leq\alpha} 1$} appears in $q(y)$ in ${<}\lambda^x$ steps. 
Otherwise $\lambda^y<\lambda^x$, so $p(y)$ will halt in ${<}\lambda^y$ and therefore ${<}\lambda^x$ steps. 
Clearly 
$\lambda^x<\sigma$. 
\end{proof}

We call an ittm-program \emph{total} if it halts for every input. 
We are now ready to prove the main results of this section.

\begin{theorem}
\label{countable decision times}
The suprema of countable decision times of (a) total programs and of (b) decidable sets equal $\sigma$. 
\end{theorem}
\begin{proof} 
Given Theorem \ref{sigma upper set clockable bound}, is remains to show that $\sigma$ is a strict upper bound for countable decision times of total programs. 
Suppose that $p$ is total and has a countable decision time. 
Since $\exists \alpha<\omega_1\ \forall x\ p(x){\downarrow}^{\leq\alpha}$ is a $\Sigma^1_2$ statement, this holds in $L$ by Shoenfield absoluteness. 
Since $L_\sigma\prec_{\Sigma^1_2} L$, there is some $\alpha<\sigma$ such that $\forall x\ p(x){\downarrow}^{\leq\alpha}$ holds in $V$, as required. 
\end{proof}



\subsection{Quick recognising}
\label{subsection - Quick recognising}  

The lost melody theorem, {\em i.e.}, the existence of recognisable, but not writable reals in \cite[Theorem 4.9]{hamkins2000infinite} shows that the recognisability strength of ittm's goes beyond their writability strength. 
It thus becomes natural to ask whether this result still works with bounds on the time complexity.
If a real $x$ can be written in $\alpha$ many steps, then it takes at most $\alpha+\omega+1$ many steps to recognise $x$ by simply writing $x$ and comparing it to the input. 
Can it happen that a writable real can be semirecognised much quicker than it can be written? 
The next lemma shows that this is impossible. 

\begin{lemma}{\label{bounded writing time}} 
Suppose that $p$ recognises $x$ and $p(x)$ halts at time $\alpha$. Then: 
\begin{enumerate-(1)} 
\item 
$x\in L_\beta$ for some $\beta< \alpha^\oplus$. 
\item
 $x$ is writable from any real coding $\beta$ in time less than $\beta^\oplus$ steps. 
 If $\beta$ is clockable $x$ is simply writable  in time less than $\beta^\oplus$.

\end{enumerate-(1)} 
\end{lemma}
\begin{proof} 
The first claim holds by Lemma \ref{recognisable reals appear quickly}. 
For the second claim,
note that there is an algorithm that writes a code for $\beta$ in at most $\beta$ many steps by the quick writing theorem  \cite[Lemma 48]{welch2009characteristics}. 
One can therefore write codes for $L_\beta$ and any element of $L_\beta$ in less than $\beta^\oplus$ many steps. 
\end{proof} 


\subsection{Gaps in the decision times}
\label{section - gaps} 

It is well known that there are gaps in the set of halting times of ittm's (see \cite[Section 3]{hamkins2000infinite}). 
We now show that the same is true for semidecision times of total programs and thus of sets. 



A \emph{gap} in the semidecision times of programs is an interval that itself contains no such times, but is bounded by one. 



\begin{theorem} 
\label{uniform running time bound} 
For any $\alpha<\sigma$, there is a gap below $\sigma$ of length at least $\alpha$ in the semidecision times of programs. 
\end{theorem} 
\begin{proof} 
Consider the $\Sigma^1_2$-statement \emph{``there is an interval $[\beta,\gamma)$ strictly below $\omega_1$ of length $\alpha$ such that for all programs $p$, there is (i) a real $y$ such that $p(y)$ halts later than $\gamma$, or (ii) for all reals $y$ such that $p(y)$ halts, it does not halt within $[\beta,\gamma)$''}. 
This statement holds, since its negation implies that any interval $[\beta,\gamma)$ strictly below $\omega_1$ of length $\alpha$ contains the decision time of a program. 
Since $L_\sigma\prec_{\Sigma_1}L$ by Lemma \ref{versions of sigma}, such an interval exists below $\sigma$. 
\end{proof}

Note that we similarly obtain gaps below $\tau$ of any length $\alpha<\tau$ by replacing $\sigma$ by $\sigma_\alpha$.

\section{Semidecision times}

In this section, we shall determine the supremum of the countable semidecision times. 
We then study semidecision times of singletons and their complements and show that undecidable singletons of this form exists.

\subsection{The supremum of countable semidecision times}

We shall need the following auxiliary result. 

\begin{lemma} 
\label{tau is the sup of Pi1-definable ordinals} 
The supremum of $\Pi_1^{L_{\omega_1}}$-definable ordinals equals $\tau$. 
\end{lemma} 
\begin{proof} 
Suppose that $\bar{\alpha}<\tau$ is $\Sigma_2^{L_{\omega_1}}$-definable as a singleton by the formula $\psi(\bar{\alpha}) = \exists \beta\ \forall w\ \varphi(\bar{\alpha},\beta,w)$, where $\varphi$ is $\Delta_0$. 
Let $\Psi(\alpha,\beta)$ abbreviate: 
{\em
$$ ``(\alpha,\beta) \text{ is } <_{lex}\text{-least such that }\forall w\ \varphi(\alpha,\beta,w)".$$} 
Then we shall have $L_{\omega_{1}}\models \Psi(\bar{\alpha},\bar{\beta})$ for some $\bar{\beta}$. 
However then for all sufficiently large $\delta < \omega_{1}$, we have likewise 
$L_{\delta}\models \Psi(\bar{\alpha},\bar{\beta})$. 
To see this, take any $\delta$ such that for each $\alpha<\bar{\alpha}$ there is some $\beta<\delta$ with $\neg\varphi(\alpha,\beta,w)$. 
Now let $\bar{\delta}$ be least with $L_{\bar{\delta}}\models \Psi(\bar{\alpha},\bar{\beta})$. 
Note that $\tau > \bar{\delta}>\max\{\bar{\alpha},\bar{\beta}\}$. 
Then we have  a $\Pi_{1}^{L_{\omega_{1}}}$ definition of $\bar{\delta}$ as a singleton:
$$\delta= \bar{\delta} \Longleftrightarrow \exists \alpha,\beta<\delta\ [\forall w\ \varphi(\alpha,\beta,w)\wedge L_{\delta}\models \mbox{``}\Psi(\alpha,\beta)\wedge \forall \eta\ (\neg\Psi(\alpha,\beta)^{L_{\eta}})\mbox{''} ].$$
There is a bounded existential quantifier in front of the conjunction of two $\Pi_1$ formulas in $\delta$. 
This is a $\Pi_1$ definition in $\delta$ over models of $\mathsf{KP}$. 
The first conjunct guarantees that the witnessing $\alpha$ equals $\bar{\alpha}$; the second conjunct that $\beta = \bar{\beta}$.
Now $\bar{\alpha}<\bar{\delta}<\tau$ as required. 
\end{proof} 

\begin{corollary} If $V=L$ then the $\Pi^1_{2}$ singleton reals appear unboundedly below $\tau$, and $\tau=\delta^{1}_{3}$, the supremum of $\Delta^{1}_{3}$ wellorders of $\omega$.
\end{corollary}

\begin{remark}
The previous corollary is the natural analogue at one level higher of the facts that the $\Pi^{1}_{1}$-singletons appear unboundedly in $\sigma$ and the latter equals the analogously defined $\delta^{1}_{2}$. 
At this lower level, the relevant objects are absolute via Levy-Shoenfield absoluteness and the assumption $V=L$ is not needed. 
\end{remark}

We shall use the effective boundedness theorem: 

\begin{lemma}[Essentially \cite{Sp55}]
\label{effective Sigma11 boundedness} 
The rank of any $\Sigma^1_1(x)$ wellfounded relation is strictly below $\omega_1^{\ck,x}$. 
In particular, any $\Sigma^1_1(y)$ subset $A$ of $\WO$ is bounded by $\omega_1^{\ck,y}$. 
\end{lemma} 

We quickly sketch the proof for the reader. 
The proof of the Kunen-Martin theorem in \cite[Theorem 31.1]{kechris2012classical} shows that the rank of $R$ is bounded by that of a computable wellfounded relation $S$ on $\omega$. 
Since $L_{\omega_1^{\ck,x}}[x]$ is $x$-admissible, the calculation of the rank of $S$ takes place in $L_{\omega_1^{\ck,x}}$ and hence the rank is strictly less than $\omega_1^{\ck,x}$. 

While the second claim (which is essentially due to Spector) follows immediately from the first one, we give an alternative proof without use of the Kunen-Martin theorem. 
Fix a computable enumeration $\vec{p}=\langle p_n\mid n\in\omega\rangle$ of all Turing programs. 
Let $N$ denote the set of $n\in\NN$ such that 
$p_n^y$ is total and the set decided by $p_n$ codes an ordinal. 
By standard facts in effective descriptive set theory (for instance the Spector-Gandy theorem \cite{Sp59, Ga}, see also 
\cite[Theorem 5.3]{Hjorth-Vienna-notes-on-descriptive-set-theory}), $N$ is $\Pi^1_1(y)$-complete.  
In particular, it is not $\Sigma^1_1(y)$. 
Towards a contradiction, suppose that $A$ is unbounded below $\omega_1^{\ck,y}$. 
Then $n\in N$ if and only if there exist $a$ decided by $p_n$, a linear order $b$ coded by $a$ and some $c\in A$ such that $b$ embeds into $c$. 
Then $N$ is $\Sigma^1_1(y)$. 

\begin{theorem} 
\label{supremum of countable semidecision times} 
The supremum of countable semidecision times equals $\tau$. 
\end{theorem} 
\begin{proof} 
To see that $\tau$ is a strict upper bound, note that the statement \emph{``there is a countable upper bound for the decision time''} is $\Sigma_2^{L_{\omega_1}}$. 
In particular if $(\exists x\ \forall y\ \psi(x,y))^{L_{\omega_1}}$ then
$(\exists x \in L_{\tau}\ \forall y\ \psi(x,y))^{L_{\omega_1}}$.

It remains to show that the set of semidecision times is unbounded below $\tau$. 
In the following proof, we call an ordinal $\beta$ an \emph{$\alpha$-index} if $\beta>\alpha$ and some $\Sigma_1^{L_{\omega_1}}$ fact with parameters in $\alpha\cup\{\alpha\}$ first becomes true in $L_\beta$. 
Thus $\sigma_\alpha$ is the supremum of $\alpha$-indices. Any such $\sigma_{\alpha}$, like $\sigma$, is an admissible limit of admissible ordinals.

Suppose that $\nu$ is $\Pi_1^{L_{\omega_1}}$-definable. (There are unboundedly many such $\nu$ below $\tau$ by Lemma  \ref{tau is the sup of Pi1-definable ordinals}). 
Fix a $\Pi_1$-formula $\varphi(u)$ defining $\nu$. 
We shall define a $\Pi^1_1$ subset $A=A_\nu$ of $\WO$. 
$A$ will be bounded, since for all $x\in A$, $\alpha_x$ will be a $\bar{\nu}$-index for some $\bar{\nu}\leq\nu$ and hence $\alpha_x<\sigma_\nu$. 

For each $x\in \WO$, let $\nu_x$ denote the least ordinal $\bar{\nu}<\alpha_x$ with $L_{\alpha_x}\models \varphi(\bar{\nu})$, if this exists. 
Let $\psi(u)$ state that \emph{``$\nu_u$ exists and $\alpha_u$ is a $\nu_u$-index''}. 
Let $A$ denote the set of $x\in\WO$ which satisfy $\psi(x)$. 
Clearly $A$ is $\Pi^1_1$.

\begin{claim*} 
\label{lower bound for decision time} 
The decision time of $A_\nu$ equals $\sigma_\nu$. 
Furthermore, for any ittm that semidecides $A_\nu$, the order type of the set of halting times for real inputs is at least $\sigma_\nu$. 
\end{claim*} 
\begin{proof} 
The definition of $A_\nu$ yields an algorithm to semidecide $A_\nu$ in time $\sigma_\nu$. 
Now suppose that for some $\gamma<\sigma_\nu$, there is an ittm-program $p$ that semidecides $A$ with decision time $\gamma$. 
Let $g$ be $\Col(\omega,\gamma)$-generic over $L_{\sigma_\nu}$ in that it meets all dense sets of this partial order that are elements of $L_{\sigma_{\nu}}$.  Let $x_g\in L_{\sigma_\nu}[g]$ be a real coding $g$. 
 Such genericity preserves the admissibility of ordinals in the interval $(\gamma,\sigma_{\nu})$, in that for such ordinals $\tau$, $L_{\tau}[g]$, and {\em a fortiori} $L_{\tau}[x_{g}]$, is an admissible set. As we have observed $\sigma_{\nu}$ is a limit of admissibles, and thus $\gamma < \omega_1^{\ck,x_{g}}<\sigma_\nu$.  However $A$ is $\Sigma^1_1(x_g)$, since $x\in A$ holds if and only if there is a halting computation $p(x)$ of length at most $\gamma$. 
By Lemma \ref{effective Sigma11 boundedness}, $A$ is bounded by $\omega_1^{\ck,x_g}$. 
This contradicts the definition of $A$, as it is unbounded in $\sigma_\nu$. 

For the second sentence of the claim, construct a strictly increasing sequence of halting times of length $\sigma_\nu$ by $\Sigma_1$-recursion over $L_{\sigma_\nu}$. 
It is unbounded in $\sigma_\nu$ by the first claim, hence its length is $\sigma_\nu$. 
\end{proof} 
This proves Theorem \ref{supremum of countable semidecision times}. 
\end{proof}

\subsection{Semirecognisable reals} 
We have essentially completed the calculation of the supremum of semidecision times of singletons. 
The upper bound follows from Lemma \ref{versions of sigma} and the lower bound from Lemma \ref{recognisable reals appear quickly}. 

\begin{theorem} 
\label{supremum of semidecision times of semirecognisable singletons} 
The supremum of semidecision times of singletons equals $\sigma$. 
\end{theorem} 

To see that this does not follows from Lemma \ref{sigma upper set clockable bound}, note that semirecognisable, but not recognisable reals exist by \cite[Theorem 4.5.4]{carl2019ordinal}. 
In fact, we shall obtain a stronger result via the next lemma. 

\begin{lemma} \
\label{semirecognisables are closed under writability} 
\begin{enumerate-(1)} 
\item 
\label{semirecognisables are closed under writability 1} 
If $x$ is  semirecognisable  and $y=_{\infty}x$, then $y$ is semirecognisable. 
\item 
\label{semirecognisables are closed under writability 2} 
If $x$ is a \emph{fast} real, that is $x\in L_{\lambda^{x}}$, then every $y=_{\infty} x$ is similarly fast.
\end{enumerate-(1)} 
\end{lemma} 
\begin{proof} 
\ref{semirecognisables are closed under writability 1}: 
Suppose $x$ is semirecognisable via the program $p$. Let $q(x)\downarrow y$ and $r(y)\downarrow x$. 
The following program semirecognises $y$. 
On input $\bar y$, compute $r(\bar y)$ and if this halts with output $\bar x$, then compute $p(\bar x)$. If the latter halts 
(and so $\bar x = x$), we perform $q(\bar x)\downarrow y$ and check that $y= \bar y$. If so we halt with output $1$ and diverge otherwise. 

\ref{semirecognisables are closed under writability 2}: 
Note that $y=_\infty x $ implies $ \lambda^{y}=\lambda^{x}$. So $y\in L_{\lambda^{x}}[x] = L_{\lambda^{x}}$. 
\end{proof}

\begin{theorem} Let $x$ be any real. 
\label{semidecidable undecidable reals} 
\begin{enumerate-(1)} 
\item 
\label{semidecidable undecidable reals - part 1} 
No real in $L_{\Sigma^{x}}[x]\setminus L_{\lambda^{x}}[x]$ is recognisable.
\item 
\label{semidecidable undecidable reals - part 2} 
No real in $L_{\zeta^{x}}[x]\setminus L_{\lambda^{x}}[x]$ is semirecognisable. 
\item 
\label{semidecidable undecidable reals - semirecognisable}
Suppose that $y$ is both fast\footnote{For the definition of \emph{fast}, see Lemma \ref{semirecognisables are closed under writability}.} and semirecognisable, $x\leq
_{\infty}y$ and 
 $\lambda^{x}=\lambda^{y}$. 
 Then $x$ is semirecognisable. 
\item 
\label{semidecidable undecidable reals - part 4} 
All reals in $L_{\Sigma}\setminus L_{\zeta}$ are semirecognisable.
\end{enumerate-(1)} 
\end{theorem}
\begin{proof} 
\ref{semidecidable undecidable reals - part 1}
Suppose that $p$ recognises $y\in L_{\Sigma^{x}}[x]$. 
We run a universal ittm $q$ with oracle $x$ and run $p(z)$ on each tape contents $z$ produced by $q$. 
Once $p$ is successful, we have found $y$ and shall write it on the output tape and halt.  Hence  $y\in L_{\lambda^{x}}[x]$.

\ref{semidecidable undecidable reals - part 2}
Suppose that $p$ eventually writes $y$ from $x$ and $q$ semirecognises $y$. 
We run $p(x)$ and in parallel $q(z)$, where $z$ is the current content of the output tape of $p$. 
Whenever the latter $z$ changes, the run of $q(z)$ is restarted. 
When $q(z)$ halts, output $z$ and halt. 
To see that this algorithm writes $y$, note that the output of $p(x)$ eventually stabilises at $y$, so $q(y)$ is run and $y$ is output when this halts. Hence $y\in L_{\lambda^{x}}[x]$.

\ref{semidecidable undecidable reals - semirecognisable} 
Since $y$ is fast, we have $y\in L_{\lambda^{y}}$; as $\lambda^{y}=\lambda^{x}$, it follows that $y\in L_{\lambda^{x}}$. Then $y\leq_{\infty}x$ by Lemma \ref{L_lambdax is writable-invariant}, and since $x\leq_{\infty}y$, we get $x=_{\infty}y$. By part (1) of Lemma \ref{semirecognisables are closed under writability}, $x$ is semirecognisable. 

\ref{semidecidable undecidable reals - part 4}
Take any $x\in L_\Sigma \setminus L_\zeta$. 

We first show that $\lambda^{x}>\zeta$.\footnote{This argument is from \cite[Theorem 2.6 (3)]{W}.}  
Assume $\lambda^{x}\leq\zeta$. 
Since $\Sigma\leq\Sigma^{x}$, we have 
$L_{\Sigma^{x}}[x]\models x\in L$. 
By $\Sigma_1$-reflection $L_{\lambda^{x}}[x]\models x\in L$ and hence $x\in L_\zeta$. 
But this contradicts the choice of $x$. 
 
We now show that $\lambda^{x}>\Sigma$. 
Since $L_{\lambda^{x}}[x]\prec_{\Sigma_{1}}L_{\Sigma^{x}}[x]$ and 
 $L_{\zeta}$ is the maximal proper  $\Sigma_{1}$-substructure of $L_{\Sigma}$, 
 we must have $\Sigma<\Sigma^{x}$. 
 The existence of a $\Sigma_2$-extendible pair reflects to $L_{\lambda^x}$ and hence $\Sigma<\lambda^x$. 

  Let $y$ denote the $<_L$-least code of $L_{\Sigma}$. 
  Clearly $y$ is recognisable via first order properties of $L_\Sigma$. 
We claim that $y$ is fast. 
Since $y\in L_{\lambda^{x}}[x]$ and $x\in L_{\lambda^{y}}[y]$, we have ${y}=_{\infty}{x}$ and $\lambda^x=\lambda^y$ by Lemma \ref{L_lambdax is writable-invariant}. 
Thus $\lambda^y>\Sigma$ by the previous argument and $y \in L_{\Sigma+1}\subseteq  L_{\lambda^{y}}$ as required. 
The result then follows from \ref{semidecidable undecidable reals - semirecognisable}.
   \end{proof}

 We remark that some requirement on $\lambda^{x}$ is needed in \ref{semidecidable undecidable reals - semirecognisable}.
 To see this, take any Cohen generic $x \in L_{\lambda^{y}}$ over $L_{\Sigma}$, where $y$ denotes the $<_L$-least code for $\Sigma$. (In this case, we have $\lambda^x=\lambda<\lambda^y$, so that the condition $\lambda^x=\lambda^y$ is violated.) 
Since $y\in L_{\Sigma+1}\subseteq L_{\lambda^y}$, $y$ is fast. Moreover, since $x\in L_{\lambda^{y}}$, we have $x\leq_{\infty}y$. Finally, $y$ is recognisable (and hence, a fortiori, semirecognisable) by testing whether it is the $<_L$-minimal code of the minimal $L$-level that has a proper $\Sigma_2$-elementary submodel.
Thus, the other assumptions of Theorem \ref{semidecidable undecidable reals - semirecognisable} are satisfied. 
We claim that $x$ is not semirecognisable by a program $p$. 
Otherwise $L_{\Sigma}[x]\models p(x){\downarrow}^\alpha$ for some $\alpha < \lambda^{x}$. 
This statement is forced over $L_{\Sigma}$ for the Cohen real, and we can take two incompatible Cohen reals over $L_\Sigma$ for which $p$ would have to halt. 

\subsection{Cosemirecognisable reals}

Here we study semidecision times for the complements of cosemirecognisable reals. 
We shall call them \emph{cosemidecision times}. 

We determine the supremum of cosemidecision times of singletons. 
To this end, we shall need an analogue to Lemma \ref{recognisable reals appear quickly}. 
It will be used to show that
any countable cosemidecision time of a program $p$ is strictly below $\sigma$.

\begin{lemma} 
\label{cosemirecognisable reals appear quickly} 
If $p$ cosemirecognises $x$ and $p(x)$ has a final loop of length  $\leq \alpha$, then 
\begin{enumerate-(1)} 
\item 
\label{cosemirecognisable reals appear quickly 1} 
$\alpha < \sigma$. 
\item 
\label{cosemirecognisable reals appear quickly 2} 
$x\in L_{\alpha^\oplus}$. 
\end{enumerate-(1)} 
\end{lemma}

\begin{proof} 
\ref{cosemirecognisable reals appear quickly 1} 
Suppose that $p$ cosemirecognizes $x$.  Then $x$ is the unique $y$ such that $p(y)$ loops. 
The statement that $p(y)$ loops for some $y$ is a true $\Sigma_1$-statement and it therefore holds in $L_\sigma$. 
By uniqueness, $x\in L_\sigma$, and further the length of that final loop is some $\alpha <\sigma$.

\ref{cosemirecognisable reals appear quickly 2} 
Let $M=L_{\alpha^\oplus}$ and take any $\Col(\omega,\alpha)$-generic filter $g\in V$ over $M$. 
Since $\Col(\omega,\alpha)$ is a set forcing in $M$, $M[g]$ is admissible. 
Let $y\in M[g]\cap \WO$ be a real coding $\alpha$. 
As in the proof of Lemma \ref{recognisable reals appear quickly}, it suffices to show that $x\in M[g]$.  Following that proof, set $R(z)$ if ``$\exists h\ [ h $ {\em  codes a  sequence of computation snapshots in $p(z)$, along the ordering $y$, 
of a 
computation of length 
$\alpha$ with a final loop}]''.
The rest of the agument is identical as $z=x$ is the only possible solution to $R(z)$.
\end{proof}

\begin{theorem} 
\label{supremum of cosemidecision times of reals} 
The supremum of cosemidecision times of 
reals equals $\sigma$. 
\end{theorem} 
\begin{proof} 
We first show that $\sigma$ is an upper bound. 
Suppose that $p$ cosemirecognises $x$. 
We define a new program $r$ which will cosemirecognize $x$ in less than $\sigma$ steps. The program $r$ will work on input $y$ by simultaneously running $p$ and the following program $q$, and halting  as soon as $p(y)$ or $q(y)$ halts. 
The definition of $q(y)$ is based on the machine considered in \cite{friedman2007two}, 
which writes the $\Sigma_2(y)$-theories of $J_\alpha[y]$ in its output, successively for $\alpha$. 
Note that the $\Sigma_2(y)$-theory of $J_\alpha[y]$ appears in step $\omega^2\cdot (\alpha +1)$. 
 (This is the reason for the choice of this specific program.)  
$q(y)$ searches within these theories for two writable reals relative to $y$: a real $z$ and a real coding an ordinal $\alpha$ such that $p(z)$ has a final loop of length at most $\alpha$ (in particular, it does not converge). 
Note that such a loop occurs by time $\Sigma^z$, if it occurs at all (see \cite[Main Proposition]{welch2000length} or \cite[Lemma 2]{welch2009characteristics}). If such reals are found, we check whether $z=y$; if that is the case, $r$ runs into a loop. Otherwise, $r$ halts.

It is not hard to see that $r$ cosemirecognises $x$: If $y=x$, then $p(y)$ will diverge. Moreover, $q(y)$ will either (i) eventually produce $x$ and a theory witnessing the fact that $p(x)$ loops, find that $x=y$ and diverge, or (ii) never produce such a theory and thus diverge while looking for it; in both cases $q(y)$ diverges. On the other hand, if $y\neq x$, then, by definition of $p$, $p(y)$, and thus $r(y)$, will halt.

It now suffices to show that the semidecision time of $r$ is at most $(\Sigma^x)^{\oplus}$, which is smaller than $\sigma$. 
Since $r(x)$ diverges, we can assume that $y\neq x$. 
We now consider two cases. 

First, if $\lambda^y\leq\Sigma^x$, then $p(y)$ halts at time $\Sigma^x$ or before. Therefore, $r(y)$ will also halt in $<(\Sigma^x)^{\oplus}$ many steps. 

Now suppose that $\lambda^y>\Sigma^x$. Since $\lambda^{y}$ is a limit of admissible ordinals, it follows that $\lambda^{y}>(\Sigma^x)^{\oplus}$. 
By Lemma \ref{cosemirecognisable reals appear quickly}, $x\in L_{(\Sigma^x)^\oplus}$. 
By the definition of $q$, the statement \emph{``there exists $z$ such that the length of $p(z)$'s loop is  $\leq \alpha$''} appears in the computation of $q(y)$ in strictly less than $(\Sigma^x)^\oplus$ steps. By definition of $p$, we will have $z=x$, and since we are assuming $y\neq x$, we will have $y\neq z$, so that, by definition of $r$, the computation $r(y)$ again halts in $<(\Sigma^x)^{\oplus}$ many steps. 


It remains to show that $\sigma$ is minimal. 
Towards a contradiction, suppose that $\beta<\sigma$ is a strict upper bound for the cosemidecision times. 
We can assume that $x_\beta$ is recognisable, for instance by taking $\beta$ to be an index. 
Suppose that $x_\beta$ is recognised by an algorithm with decision time $\alpha$. 
Note that $\alpha<\sigma$ by Theorem \ref{sigma upper set clockable bound}. 
Take $\gamma\geq (\alpha+\beta)^\oplus$ such that $x_\gamma$ is recognisable. 
Then $x_\gamma\oplus x_\beta$ is recognisable.

We claim that $x_\beta\oplus x_\gamma$ is not cosemirecognisable by an algorithm with decision time strictly less than $\beta$. 
So suppose that $p$ such an algorithm. 
We shall describe an algorithm $q$ that semirecognises $x_\gamma$ in at most $\alpha+\beta+1$ steps. 
This contradicts Lemma \ref{recognisable reals appear quickly}. 
Note that $x_\beta$ is coded in $x_\gamma$ by a natural number $n$. 
The algorithm $q$ extracts the real $x$ coded by $n$ from the input $y$. 
It then decides whether $x= x_\beta$, taking at most $\alpha$ steps, and diverges if $x\neq x_\beta$. 
If $x=x_\beta$, we run $p(x\oplus y)$ for $\beta$ steps and let $q$ halt if and only if $p(x\oplus y)$ fails to halt before time $\beta$. 
Then $q(y)$ halts if and only if $y=x_\gamma$. 
\end{proof} 


The previous result would follow from Theorem \ref{sigma upper set clockable bound} if every cosemirecognisable real were recognisable. 
However, the next result disproves this and thus answers \cite[Question 4.5.5]{carl2019ordinal}. 


\begin{theorem} 
The cosemirecognisable, but not recognisable, reals appear cofinally in $L_\sigma$, that is, their $L$-ranks are cofinal in $\sigma$.
\end{theorem} 
\begin{proof} 
Let $\xi$ be an index, $x=x_\xi$ and $y=x_{\lambda^x}$. 
Since $y\in L_{\zeta^x}\setminus L_{\lambda^x}$, it is not semirecognisable by Lemma \ref{semidecidable undecidable reals} \ref{semidecidable undecidable reals - part 2}. 
We claim that $y$ is cosemirecognisable. 
We shall assume that $\lambda^x=\lambda$; the general case is similar. 

For an input $z$, first test whether it fails to be a code for a wellfounded $L_\alpha$; if it is such a code, then check if it has an initial segment which itself has a $\Sigma_2$-substructure (equivalently $\alpha\geq \Sigma$); if it fails this test, then check if it has a $\Sigma_1$-substructure (if it does then $\alpha \neq \lambda$). 
If it fails this last point, then $z$ is a code for an $L_\alpha$ with $\alpha \leq \lambda$. 
Check if $z$ fails to be the $<_L$-least code for $L_\alpha$. 
Furthermore, run a universal machine and check whether some program halts beyond $\alpha$. 
\end{proof}




\section{Sets with countable decision time} 

Any set $A$ with countable semidecision time $\alpha$ is $\Sigma^1_1$ in any code for $\alpha$. 
If this is witnessed by a program $p$ and and $y$ is a code for $\alpha$, then 
$x\in A$ if and only if {\em there exists a halting computation of $p(x)$ 
along the ordering coded by $y$}.  Thus $A$ is  a $\Sigma^{1}_{1}(y)$ set. 
Similarly, any set with a countable decision time is both $\Sigma^1_1$ and $\Pi^1_1$ in any code for $\alpha$, and is thus Borel. 
The next result shows that for both implications, the converse fails. 
This complements Lemma \ref{long decision times}. 

\begin{theorem} 
\label{Borel set with uncountable decision time} 
There is a cocountable open decidable set $A$ that is not semidecidable in countable time. 
\end{theorem} 
\begin{proof} 
Let $\vec{\varphi}=\langle \varphi_n\mid n\in\omega\rangle$ be a computable enumeration of all $\Sigma_1$-formulas with one free variable. 
Let $B$ denote the discrete set of all $0^n{}^\smallfrown \langle1\rangle^\smallfrown x$, where $x$ is the $<_L$-least code for the least $L_\alpha$ where $\varphi_n(x)$ holds. 
Let further $p$ denote an algorithm that semidecides $B$ as follows.
First test if the input equals $0^\omega$ and halt in this case. 
Otherwise, test if the input is of the form $0^n{}^\smallfrown \langle1\rangle^\smallfrown x$, run a wellfoundedness test for $x$, which takes at least $\alpha$ steps for codes for $L_\alpha$, and then test whether $\alpha$ is least such that $\varphi_n(x)$ holds in $L_\alpha$. 
The decision time of $p$ is at least $\sigma$. 
Moreover, it is countable since $B$ is countable. 

Let $A$ denote the complement of $B$. 
Towards a contradiction, suppose that $q$ semidecides $A$ in countable time. 
Let $r$ be the decision algorithm for $B$ that runs $p$ and $q$ simultaneously. 
Then $r$ has a countable decision time $\alpha$ and by $\Sigma^1_2$-reflection, we have $\alpha<\sigma$. 
But this is clearly false, since $p$'s decision time is at least $\sigma$. 
\end{proof}


\section{Open problems} 

The above results for sets of reals also hold for Turing machines with ordinal time and tape with virtually the same proofs, while the results for singletons do not. 
If we restrict ourselves to sets of natural numbers, then the suprema of decision and semidecision times 
equal $\lambda$, the supremum of clockable ordinals. 

In the main results, we determined the suprema of various decision times, but we have not characterised the underlying sets. 

\begin{question} 
Is there a precise characterisation in the $L$-hierarchy of sets and singletons with countable decision, semidecision and cosemidecision times? 
\end{question}

We would like to draw a further connection with classical results in descriptive set theory. 
A set of reals is called \emph{thin} if it does not have an uncountable closed subset. 
It can be shown that $\{x\mid x \in L_{\lambda^x}\}$ is the largest thin semidecidable set 
(see \cite[Definition 1.6]{W}; this unpublished result of the third-listed author should appear in \cite{countableranks}).
We ask if the same characterisation holds for eventually semidecidable sets. 

\begin{question} 
Is $\{x\mid x \in L_{\lambda^x}\}$ the largest thin eventually semidecidable set? 
\end{question}

In particular, is every eventually semidecidable singleton an element of $L_{\lambda^x}$ (equivalently $L_{\Sigma^x}$)? 
An indication that these statements might be true is that one can show an analogous statement for null sets instead of countable sets: the largest ittm-semidecidable null set equals the largest ittm-eventually semidecidable null set. 
Related to Theorem \ref{Borel set with uncountable decision time}, it is natural to ask whether a thin semidecidable set can have halting times unbounded in $\omega_1$. 


We did not study specific values of decision times in this paper. 
Note that this is an entirely different type of problem, since it is sensitive to the precise definition of ittm's. 
Regarding Section \ref{section - gaps}, we know from \cite[Theorem 8.8]{hamkins2000infinite} that admissible ordinals are never clockable and from \cite[Theorem 50]{welch2009characteristics} that any ordinal that begins a gap in the clockable ordinals is always admissible. 
One can ask if an analogous result holds for decision times. 

\begin{question}
Is an ordinal that begins a gap of the (semi-)decision times always admissible? 
\end{question}

One can also ask about the decision times of single sets. 
By the Bounding Lemma \cite[Theorem 8]{W08}, no decidable set $A$ has a countable admissible decision time. 
However, the semidecision time of a decidable set can be admissible. 
To see this, note that the set of indices as in the proof of Theorem \ref{supremum of countable semidecision times} can be semidecided by a program with semidecision time $\sigma$. 
We do not know if this is possible for a set that is not decidable. 






\bibliographystyle{alpha}
\bibliography{References} 

\end{document}